\documentclass[10pt]{article}
\usepackage{amsmath,amsthm,amsfonts,amssymb,url}
\usepackage{authblk}
\usepackage{color}

\newtheorem{theorem}{Theorem}[section]
\newtheorem{lemma}[theorem]{Lemma}
\newtheorem{corollary}[theorem]{Corollary}
\newtheorem{remark}{Remark}
\newtheorem{example}{Example}[section]

\newcommand{\zed}{\ensuremath{\mathbb{Z}}} 

\title{On Resolvable Golomb Rulers, Symmetric Configurations and Progressive Dinner Parties}
\author[1]{Marco Buratti\thanks{This work has been performed under the auspices
of the G.N.S.A.G.A.\ of the C.N.R.\ (National Research
Council) of Italy}}
\author[2]{Douglas R.\ Stinson\thanks{D.R.\ Stinson's research is supported by  NSERC discovery grant RGPIN-03882.
}}
\affil[1]{Dipartimento di Matematica e Informatica, 
Universit\`{a} di Perugia, 06123, Perugia, Italy}
\affil[2]{David R.\ Cheriton School of Computer Science, University of Waterloo,
Waterloo, Ontario, N2L 3G1, Canada}

\date{\today}

\begin{document}
\maketitle

\begin{abstract}
We define a new type of Golomb ruler, which we term a \emph{resolvable Golomb ruler}.
These are Golomb rulers that satisfy an additional ``resolvability'' condition that allows them to generate \emph{resolvable symmetric configurations}. The resulting configurations give rise to \emph{progressive dinner parties}. In this paper, we investigate existence results for resolvable Golomb rulers and their application to the construction of resolvable symmetric configurations and progressive dinner parties. In particular, we determine the existence or nonexistence of all possible resolvable symmetric configurations and progressive dinner parties having block size at most $13$, with nine possible exceptions. For arbitrary block size $k$, we prove that these designs exist if the number of points is divisible by $k$ and at least $k^3$.
\end{abstract}

\section{Introduction}

A \emph{Golomb ruler} of \emph{order} $k$ is a set of $k$ distinct integers, say
$x_1 < x_2 < \cdots < x_k$, such that all the differences $x_j - x_i$ ($i \neq j$) are distinct. The \emph{length} of the ruler is $x_k - x_1$. For a survey of constructions of Golomb rulers, see \cite{Drakakis}.

We should note that Golomb rulers have been studied under various names, including
\emph{Sidon sets}, \emph{sum-free sets} and  \emph{$B_2$-sequences}. 

Any translate of a Golomb ruler is again a Golomb ruler. So, if we wish, we can assume without loss of generality that $x_1 = 0$. 

In this paper we define and study a new kind of Golomb ruler. 
A Golomb ruler of order $k$ is \emph{resolvable} if 
$x_j - x_i \not\equiv 0 \bmod k$ for all $i \neq j$. Equivalently, the set
$\{x_1, x_2 , \dots , x_k\}$ covers all $k$ residue classes in $\zed_k$.
We will use the notation RGR$(k,L)$ to denote a resolvable Golomb ruler of order $k$ and length $L$.

An RGR$(k,L)$ is \emph{optimal} if there does not exist an RGR$(k,L')$ with $L' < L$.
Some examples of optimal RGR$(k,L)$ are presented in Table \ref{smallk.tab}. These were all found using a simple exhaustive backtracking algorithm.  
Let $L^*(k)$ ($L^*_R(k)$, resp.) denote the length of an optimal  Golomb ruler
(optimal resolvable Golomb ruler, resp.). Table \ref{smallk.tab} also lists the values
$L^*(k)$  and $L^*_R(k)$ for small orders. 
The values $L^*(k)$ are 
all found in \cite[\S19.2]{HCD}.

\begin{table}
\[
\begin{array}{r|r|r|l}
k & L^*_R(k) & L^*(k) & \multicolumn{1}{|c}{\text{RGR}(k,L)}\\ \hline
3 & 5 &  3& \{0 , 1 , 5\}\\
4 & 9 &  6& \{0 , 2 , 3 , 9\}\\
5 & 14 &  11&  \{0 , 1 , 8, 12 ,14\}\\
6 & 20 &  17&  \{0 , 1 , 3 ,11 ,16 ,20\}\\
7 & 31 &  25& \{ 0 , 1 , 9 ,12 ,25 ,27 ,31\}\\
8 & 45 &  34& \{ 0 , 1 , 3 ,15 ,28 ,34 ,38 ,45\}\\
9 & 58 &  44& \{ 0 , 1,  3 , 7 ,20 ,32 ,42 ,53 ,58\}\\
 10 & 69 & 55& \{0  ,1  ,3 , 7 ,18 ,26 ,42 ,55, 64, 69\}\\
11 & 87 & 72& \{0 , 4 , 6 , 9 ,27, 41 ,51, 67 ,79, 80, 87\}\\
12 & 107 &  85& \{0  , 1 ,  6 , 15 , 17 , 38 , 46 , 56 , 81 ,100 ,103 ,107\}\\
13 & 132 &  106 &  \{0 ,  1 ,  4 , 18 , 37,  46  ,48 , 71 , 77 ,112 ,120, 127, 132\}
\end{array}
\]
\caption{Some optimal resolvable Golomb rulers}
\label{smallk.tab}
\end{table}

For future use, we define a {modular Golomb ruler}. A \emph{$(v,k)$-modular Golomb ruler}
(or $(v,k)$-MGR) is is a set of $k$ distinct integers, \[0 \leq x_1 < x_2 < \cdots < x_k \leq v-1,\] such that all the differences $x_j - x_i \bmod v$ ($i \neq j$) are distinct elements of $\zed_v$. We define length and order as as before.  It is obvious that a modular Golomb ruler is automatically a Golomb ruler.

Suppose $v \equiv 0 \bmod k$.  
Then we can define a  \emph{$(v,k)$-resolvable modular Golomb ruler}
(or $(v,k)$-RMGR) to be a $(v,k)$-MGR where the $k$ elements cover all $k$ residue classes modulo $k$. 

For example, it can be verified that $\{0 , 1 , 5\}$, which is an RGR$(3,5)$, is  
a $(12,3)$-RMGR, but not a $(6,3)$-RMGR or $(9,3)$-RMGR.

The rest of this paper is organized as follows. Necessary conditions and existence results  for resolvable Golomb rulers are discussed in Section \ref{RGR.sec}. Section \ref{RSC.sec} introduces resolvable symmetric configurations 
and proves an equivalence between a certain class of these designs and affine planes.
Section \ref{CC.sec} presents results on cyclic  resolvable symmetric configurations and discusses their relationship to resolvable Golomb rulers. Existence results are also given in this section, where we obtain almost complete results for block sizes that are at most $13$, as well as a general existence result that holds for all block sizes.
Section \ref{PDP.sec} points out the equivalence of resolvable symmetric configurations and progressive dinner parties. Section \ref{summary.sec} is a short summary, in which we also discuss some problems for future research.

\section{Results on Resolvable Golomb Rulers}
\label{RGR.sec}

Here is a simple necessary counting condition for the existence of a resolvable Golomb ruler.

\begin{lemma}
\label{bound.lem}
Suppose there is an RGR$(k,L)$. 
Then 
\begin{equation}
\label{bound.eq}
 L - \left\lfloor \frac{L}{k} \right\rfloor \geq \binom{k}{2}.
 \end{equation}
\end{lemma}

\begin{proof}
Let $\{x_1, x_2 , \dots , x_k\}$  be a resolvable Golomb ruler of length $L$,
where $x_1 < x_2 < \cdots < x_k$. There are $\binom{k}{2}$ differences $x_j  - x_i$ with $j > i$.  These differences are distinct integers in the set \[ \{1, \dots , L\} \setminus \left\{ k, 2k, \dots ,
\left\lfloor \frac{L}{k} \right\rfloor k \right\}.\] Since this set has cardinality
 \[ L - \left\lfloor \frac{L}{k} \right\rfloor,\]
 the result follows.
\end{proof}

\begin{theorem}
\label{bound.cor}
Suppose there is an RGR$(k,L)$. 
If $k$ is even, then  $ L \geq k^2/2 -1$, and if $k$ is odd, then  $ L \geq (k^2-1)/2$.
\end{theorem}

\begin{proof}
This is just an application of Lemma \ref{bound.lem}. 
For $k$ even, (\ref{bound.eq}) is satisfied for $L = k^2/2-1$ but not for $L = k^2/2-2$.
For $k$ odd, (\ref{bound.eq}) is satisfied for $L = (k^2-1)/2$ but not for $L = (k^2-3)/2$.
\end{proof}

However, we note that there is already a better necessary condition for the existence of an arbitrary (not necessarily resolvable) Golomb ruler. 

\begin{theorem}
\label{bound2.cor}
\cite{Lind,Dim}
Suppose there is a Golomb ruler of order $k$ and length $L$. 
Then  $L > k^2 - 2k\sqrt{k} + \sqrt{k} -2$.
\end{theorem}

We now present an interesting general existence result that makes use of a construction for modular Golomb rulers due to Ruzsa \cite{ruzsa}.

\begin{theorem}
\label{ruzsa.thm}
Suppose $p$ is prime. Then there is an RGR$(p-1,L)$, where $L \leq p^2 - 2p$. 
\end{theorem}

\begin{proof}
We use the construction given by Ruzsa in \cite[Theorem 4.4]{ruzsa}.
Let $g$ be a primitive root modulo $p$. For $1 \leq i \leq p-1$, let $a_i$ be the solution to the two congruences 
\[ a_i \equiv i \bmod (p-1), \quad\quad a_i \equiv g^i \bmod p.\]
Denote $A = \{ a_i : 1 \leq i \leq p-1 \}$. 

Each $a_i$ has a unique (nonzero) solution modulo $p(p-1)$, so we can assume that 
$1 \leq a_i \leq p^2 - p - 1$ for $1 \leq i \leq p-1$. It is shown in 
\cite{ruzsa} that $\{a_1, \dots , a_{p-1}\}$ is a $(p^2-p,p-1)$-MGR. It is clear that the $a_i$'s cover all $p-1$ residue classes modulo $p-1$, since  $a_i \equiv i \bmod (p-1)$ for all $i$. Hence, $A$ is an RGR$(p-1,L)$, where $L \leq p^2 - p - 2$. 

However, we can improve this slightly by using a technique described in 
\cite[\S 5.7]{Dim}. 
Suppose we sort the elements in $A$ in increasing order, obtaining $B = \{b_1, \dots , b_{p-1}\}$, where
\[ b_1 < b_2 < \dots < b_{p-1}.\] 
The set of gaps between cyclically consecutive elements of $B$ is
\[ \mathcal{G} = \{ b_{i+1} - b_i : 1 \leq i \leq p-2\} \, \bigcup \, \{ b_1 - b_{p-1} \bmod p(p-1)\} .\]
Let $G$ be the maximum element in $\mathcal{G}$ and suppose $G = b_{i+1} - b_i$. If we subtract $b_{i+1}$ from every element of $B$ (modulo $p(p-1)$), then we get an RGR$(p-1,p^2-p-G)$.

Finally, the average length of a gap is $(p^2-p)/(p-1) = p$, so $G \geq p$. Therefore there exists an 
RGR$(p-1,L)$, where $L \leq p^2 - 2p$.
\end{proof}

\begin{remark}
The rulers constructed in Theorem \ref{ruzsa.thm} are ``close to'' optimal. For appropriate values of $k$, we obtain 
RGR$(k,L)$ with $L \leq k^2-1$ from this construction. On the other hand, the necessary condition from Theorem \ref{bound2.cor} is
$L > k^2 - 2k\sqrt{k} + \sqrt{k} -2$. Hence,  for any $c < 1$, there does not exist an infinite class of RGR$(k,L)$ with $L \leq ck^2$. 
\end{remark}

\begin{example}
Suppose we take $p = 11$ and we apply Ruzsa's construction with  the primitive root $g=6$. 
Then we get the set
\[ A =  \{61, 102, 73, 64, 65, 16, 107, 48, 79, 100\}.\] 
After sorting, we have
\[ B = \{16, 48, 61, 64, 65, 73, 79, 100, 102, 107\}.\]
The gaps between cyclically consecutive elements of $B$ are the elements in the set 
\[\mathcal{G} = \{32, 13, 3, 1, 8, 6, 21, 2, 5, 19\}\] (note that $16 - 107 \equiv 19 \bmod 110$). 
The maximum gap is $G = 32 = 48 - 16$, so we get an RGR$(10,78)$ by subtracting $48 \bmod 110$ from every element of $B$. The resulting ruler is 
\[0,13,16,17,25,31,52,54,59,78.\] This ruler is an RGR$(10,78)$.

Note that 
Theorem \ref{ruzsa.thm} only guarantees the existence of a ruler of length at most $99$. Of course, we will probably do better for any given value of $p$ because the gaps will all not be the same size. 
\end{example}

Two examples of small rulers that result from this construction are
$\{0,2,3,9\}$, which is an RGR$(4,9)$; and $\{0,4,9,17,19,20\}$, which is an RGR$(6,20)$. 
These two rulers turn out to be optimal resolvable rulers. 




We present some data in Table \ref{RGR.tab} that is obtained from the construction described in Theorem \ref{ruzsa.thm}, for all primes $p$ such that $5 \leq p < 100$. For each such prime $p$, we consider all primitive roots modulo $p$. For each primitive root $g$, we construct the set $B$ and then find the largest gap. This leads to RGR$(k,L)$, where $k =p-1$, for the stated values of $L$. For $k \leq 12$, we also list the length of the optimal resolvable ruler (from Table \ref{smallk.tab}).

\begin{table}[tb]
\caption{Some RGR$(k,L)$ obtained from the construction in Theorem \ref{ruzsa.thm}}
\label{RGR.tab}
\[
\begin{array}{r|r|r|r|r}
p & g & k & L & L^*_R(k)\\ \hline
                          5& 2& 4& 9 & 9\\
                          7& 3& 6& 20 & 20\\
                         11& 2& 10& 78& 69\\
                        13& 6& 12& 112 & 107 \\
                        17& 10& 16& 194\\
                        19& 14& 18& 265\\
                        23& 5& 22& 392\\
                        29& 2& 28& 607\\
                        31& 11& 30& 737\\
                        37& 18& 36& 1148\\
                        41& 6& 40& 1318\\
                        43& 12& 42& 1610\\
                        47& 35& 46& 1877\\
                        53& 33& 52& 2399\\
                        59& 39& 58& 3071\\
                        61& 17& 60& 3194\\
                        67& 11& 66& 4057\\
                        71& 59& 70& 4524\\
                        73& 53& 72& 4729\\
                        79& 63& 78& 5583\\
                        83& 8& 82& 6229\\
                        89& 63& 88& 7025\\
                        97& 23& 96& 8762
 \end{array}
\]
Note: $g$ is the primitive root modulo $p$ that is used to construct the set $A$
\end{table}

\bigskip

We now present a construction of resolvable Golomb rulers from Costas arrays. This construction yields rulers of greater length than those obtained from Theorem \ref{ruzsa.thm} ($L \approx 2k^2$ as opposed to $L \approx k^2$). However, the construction using Costas arrays can be applied for more values of $k$ (this will be discussed in more detail a bit later).

We make use of a construction of Golomb rulers from Costas arrays due to Drakakis and Rickard \cite{DR}. 
A \emph{Costas array of order $n$} consists of a set $A$ of $n$ ordered pairs in the set
$\{1, \dots , n\} \times \{1, \dots , n\}$ that satisfies the following properties:
\begin{enumerate}
\item the first co-ordinates of the $n$ points in $A$ are distinct
\item the second co-ordinates of the $n$ points in $A$ are distinct
\item for any four points $a,b,c,d \in A$, $a-b = c-d$ only if $a=c$ and $b= d$ or if $a=b$ and $c= d$. Equivalently, the $\binom{n}{2}$ vectors $a - b$ ($a,b \in A$, $a \neq b$) are distinct.
\end{enumerate}

An equivalent definition is that of a \emph{Costas permutation}. Suppose that
$f : \{1, \dots , n\} \rightarrow \{1, \dots , n\}$ is a bijection (i.e., it defines a permutation). Then $f$ is a Costas permutation  if 
\[f(i + k) - f(i) = f(j + k) -  f(j) \Rightarrow i = j \text{ or } k = 0\]
for all choices of $i,j,k$ such that $i,j,i+k,j+k \in \{1, \dots , n\}$.
A Costas permutation is constructed from a Costas array $A$ by defining 
$f(i) = j$ if and only if $(i,j) \in A$

\begin{theorem} 
\label{costas.thm}
Suppose there exists a Costas array of order $n$. Then there exists
an RGR$(n,2n^2-n-1)$.
\end{theorem}

\begin{proof}
Given a Costas array of order $n$, let 
 $f$ be the associated Costas permutation defined on $\{1, \dots , n\}$. Let $m \geq 2n-2$. 
For $1 \leq i \leq n$, define
\[ x_i = (i - 1)m + f(i).\]
It is shown in \cite{DR} that $X = \{x_i, \dots , x_n\}$ is a Golomb ruler. Clearly the length of $X$ is at most $(n-1)m +n - 1 = (n-1)(m+1)$.

Suppose we take $m = 2n$. Then $x_i \equiv f(i) \bmod n$ for $1 \leq i \leq n$. Therefore we have a resolvable Golomb ruler, because $f$ is a permutation of $\{1,\dots , n\}$. The length of this ruler is at most $(n-1)(2n+1)$. 
\end{proof}

\begin{example} A Costas array of order $4$ is given by $\{(1,2), (2,1), (3,3), (4,4)\}$.
The associated Costas permutation is defined as $f(1) = 2$, $f(2) = 1$, $f(3) = 3$ and
$f(4) = 4$. When we apply the construction described in Theorem \ref{costas.thm}, we obtain the  Golomb ruler
$\{ 2,9,19,28\}$, which is an RGR$(4,26)$.
\end{example}

Costas arrays of order $n$ are known to exist for the following values of $n$:
\begin{itemize}
\item $n = p-1$ where $p$ is prime,
\item $n = q-2$ and $n = q-3$, where $q$ is a prime power, and
\item $n \leq 29$.
\end{itemize}

\bigskip

Our final construction is a general construction of RGR$(k,L)$ for all $k$ with $L \approx
k^3/2$. This construction can be applied for any value of $k$. 

\begin{theorem}
\label{generalk.thm}
The set of integers $X=\{x_0,x_1,\dots,x_{k-1}\}$ defined by
\[x_i = \binom{i}{2} k-i,\] for $i=0,1,\dots,k-1$, is an RGR$(k,L)$ with $L=\frac{1}{2}(k+1)(k-2)^2$.
\end{theorem}
\begin{proof}
It is clear that $1=x_0-x_1$ appears as a diffrence exactly once.
Now assume that two differences from $X$, say $x_{i_1}-x_{j_1}$ and $x_{i_2}-x_{j_2}$, are both equal to an integer $d>1$.
 It is then evident that  $j_1<i_1$ and $j_2<i_2$.
 Then, by means of  elementary calculations we get
  \begin{equation}\label{EqualDifferences}
\frac{(i_1-j_1)(i_1+j_1-1)}{2}k-(i_1-j_1)=\frac{(i_2-j_2)(i_2+j_2-1)}{2}k-(i_2-j_2).
 \end{equation}
 This implies that $i_1-j_1\equiv i_2-j_2$ (mod $k$). On the other hand, both $i_1-j_1$ and $i_2-j_2$
 are non-negative integers in the set $\{0,1,\dots,k-1\}$. It necessarily follows that $i_1-j_1=i_2-j_2$.
 Replacing $i_2-j_2$ with $i_1-j_1$ in (\ref{EqualDifferences}) and simplifying, we get $i_1+j_1=i_2+j_2$.
 The two equalities  $i_1-j_1=i_2-j_2$ and $i_1+j_1=i_2+j_2$ clearly imply that the  
 two pairs $(i_1,j_1)$ and $(i_2,j_2)$ are equal. We conclude that $X$ has no repeated differences, and 
 hence it is a Golomb ruler. We have $x_i\equiv i \bmod k$ for each $i$, so  $X$ is resolvable.
 Finally, the maximum and the minimum elements in $X$ are 
 \begin{center}
 $x_{k-1}=\binom{k-1}{2}k-k+1$ \quad and \quad $x_1=-1$,
 \end{center}
 respectively. Their difference is $\frac{1}{2}(k+1)(k-2)^2$ and the assertion follows.
\end{proof}

\begin{remark}
Even for GR$(k,L)$, it seems to be difficult to give explicit constructions for all $k$ that have relatively small values of $L$. Theorem \ref{generalk.thm} is similar to \cite[Construction 2]{Dim}, but the value of $L$ in our result is approximately 50\% smaller 
than  in \cite{Dim}.
\end{remark}

\section{Resolvable Symmetric Configurations}
\label{RSC.sec}

A \emph{$(v,b,r,k)$-configuration} is a set system  $(V,\mathcal{B})$, where $V$ is a set of $v$ \emph{points} and $\mathcal{B}$ is a set of $b$ \emph{blocks}, each of which contains exactly  $k$ points, such that the following properties hold:
\begin{enumerate}
\item no pair of points occurs in more than one block, and
\item every point occurs in exactly $r$ blocks.
\end{enumerate}
It is easy to see that the parameters of a $(v,b,r,k)$-configuration satisfy the equation
$bk = vr$. 
For basic results on configurations, see \cite[\S VI.7]{HCD}.

\begin{remark} The notation ``configuration $(v_r,b_k)$'' is often used in the literature to denote
a $(v,b,r,k)$-configuration. 
\end{remark}

A $(v,b,r,k)$-configuration is \emph{symmetric} if $v = b$, which of course implies $r=k$.
We will use the notation \emph{$(v,k)$-configuration} to denote a $(v,v,k,k)$-symmetric configuration.

Suppose $v \equiv 0 \bmod k$. A $(v,b,r,k)$-configuration is \emph{resolvable} if the set of blocks can be partitioned into $r$ parallel classes, each of which consists of $v/k$ blocks that partition the set of points. 
A resolvable configuration will be denoted as $(V,\mathcal{B},\mathcal{R})$, where $V$ is the point-set, 
$\mathcal{B}$ is the block-set, and $\mathcal{R}$ is the \emph{resolution}, i.e., the set of $r$ parallel classes, as defined above. 

We note that there has been some systematic study of symmetric configurations, e.g., see
\cite{DFGMP}. There is also at least one paper on resolvable configurations, namely, \cite{Gevay}. However, we are not aware of any previous work addressing resolvable symmetric configurations, other than the results in \cite{stinson}. 

We recall a few basic results from \cite{stinson}; however, we should note that the results in \cite{stinson} were not phrased in terms of configurations.

A simple necessary condition for the existence of a resolvable symmetric configuration was given in \cite{stinson}.

\begin{theorem}
\cite[Lemma 3.1]{stinson}
\label{nec.thm}
A resolvable $(v,k)$-configuration exists only if $v \geq k^2$.
\end{theorem} 

Here are two existence results from \cite{stinson}.

\begin{theorem}
\label{k=345.thm}
\cite[Theorems 2.2 and 3.7]{stinson}
For $3 \leq k \leq 5$,  a resolvable $(v,k)$-configuration exists if and only if
$v \geq k^2$, $v \equiv 0 \bmod k$.
\end{theorem}

\begin{theorem}
\label{MOLS.thm}
\cite[Corollary 3.3]{stinson}
Suppose there are $k-1$ mutually orthogonal latin squares (MOLS) of order $w$. 
Then there is a resolvable $(kw,k)$-configuration. 
\end{theorem}

From the existence of $q-1$ MOLS of order $q$ when $q$ is a prime or a prime power, we immediately obtain the following corollary of Theorem \ref{MOLS.thm}.

\begin{corollary}
\label{MOLS.cor}
Suppose $q$ is a prime or prime power and $k \leq q$.
Then there is a resolvable $(kq,k)$-configuration. 
\end{corollary}

Theorem \ref{nec.thm} stated that  a resolvable $(v,k)$-configuration exists only if $v \geq k^2$.  In the  boundary case, when $v = k^2$, it can be shown that the configuration is equivalent to an affine plane of order $q$. This is a consequence of Bruck's Embedding Theorem \cite{Bruck}, which is a very general result. To be specific, Bruck gives a sufficient condition for a net of deficiency $d$ to be embeddable in an affine plane. The result we need concerns the embeddability of nets of deficiency $d=1$. We thought it might be of interest to give a direct proof for this special case, as opposed to relying on the proof of the general result.

\begin{theorem}
\label{AP.thm}
\cite{Bruck}
A resolvable $(k^2,k)$-configuration is equivalent to an affine plane of order $k$.
\end{theorem}

\begin{proof}
It is obvious that removing one parallel class from the resolution of an affine plane of order $k$ one gets a resolvable $(k^2,k)$-configuration. (This also follows from Theorem \ref{MOLS.thm}, because an affine plane of order $k$ is equivalent to $k-1$ MOLS of order $k$.)

Now assuming that a resolvable $(k^2,k)$-configuration exists, 
we prove that it comes from an affine plane of order $k$.
Let $V$,  $\mathcal{B}$ and $\mathcal{R}$ be the point-set, the block-set and the resolution of the given configuration.
Two points $x,y\in V$ are \emph{collinear} if they are distinct and there is a block $B\in\mathcal{B}$ containing them.
The relation of non-collinearity $\sim$ in $V$ is clearly reflexive and symmetric. Let us show that it is also transitive. 

Let $x\sim y$, $x\sim z$ and assume reductio ad absurdum that
$y\not\sim z$ so that $y$ and $z$ are distinct and there is a block $B_{0}\in \mathcal{B}$ containing both of them.
Let $\mathcal{P}_0$ be the parallel class of $\cal R$ containing $B_{0}$.
Now take a parallel class $\mathcal{P}\in\mathcal{R}\setminus\{\mathcal{P}_0\}$ and denote by $B_x(\mathcal{P})$, $B_y(\mathcal{P})$ and $B_z(\mathcal{P})$ the blocks of $\mathcal{P}$ 
containing $x$, $y$ and $z$, respectively. These three blocks are clearly pairwise distinct. 

Also, it is evident that the elements of $B_{0}\setminus\{y,z\}$ must belong to pairwise distinct blocks of $\mathcal{P}$ and none of them can be in $B_y(\mathcal{P})$ or in $B_z(\mathcal{P})$.

Thus, considering that $\mathcal{P}$ has $k$ blocks, by the pigeonhole principle, there is exactly one element $f(\mathcal{P}) \in B_{0}\setminus\{y,z\}$ belonging to $B_x(\mathcal{P})$.
Also, note that if $\mathcal{P}$ and $\mathcal{P}'$ are distinct parallel classes of $\mathcal{R}\setminus\{\mathcal{P}_0\}$, then we have $f(\mathcal{P})\neq f(\mathcal{P}')$, otherwise $B_x(\mathcal{P})$ and
$B_x(\mathcal{P}')$ would be two distinct blocks in $\cal B$ containing both the points $x$ and $f(\mathcal{P})$.

This means that the map $f : \mathcal{R}\setminus\{\mathcal{P}_0\}  \longrightarrow  B_{0}\setminus\{y,z\}$ defined by $\mathcal{P} \mapsto f(\mathcal{P})$ is injective.  
But this is a contradiction, because $|\mathcal{R}\setminus\{\mathcal{P}_0\}| = k-1$ and $|B_{0}\setminus\{y,z\}| = k-2$.

Thus, we have shown that $\sim$ is an equivalence relation. 
If $x$ is any element of $V$, then it is collinear with 
exactly $k(k-1)=k^2-k$ points. It follows that 
each equivalence class under $\sim$ has size $k$. Let  $\cal Q$ be the set of equivalence classes. Then it is clear that 
$\mathcal{R} \ \cup \ \{ \mathcal{Q} \}$ is the resolution of an affine plane of order $k$.
\end{proof}

\section{Cyclic Configurations and Golomb Rulers}
\label{CC.sec}

A $(v,b,r,k)$-configuration $(V,\mathcal{B})$ is said to be \emph{cyclic} if, up to isomorphism, $V=\zed_v$ and 
$\mathcal{B}$ is invariant under the action $x \mapsto x+1$ (mod $v$).
A resolvable $(v,b,r,k)$-configuration $(V,\mathcal{B},\mathcal{R})$ is \emph{cyclic} if, up to isomorphism, $V=\zed_v$ and $\mathcal{R}$ is invariant under the action $x \mapsto x+1$ (mod $v$).
We note that there has been some study of cyclic $(v,3)$-configurations \cite{KKP}.

We now show that a cyclic symmetric configuration is equivalent to a $(v,k)$-MGR. 
We also prove that a resolvable cyclic symmetric configuration is 
equivalent to a $(v,k)$-RMGR. 

One direction of the proof is easy, as we demonstrate in the following theorem. 

\begin{theorem} If a $(v,k)$-MGR exists, then a cyclic $(v,k)$-configuration exists. Further, if 
a $(v,k)$-RMGR exists, then a cyclic resolvable $(v,k)$-configuration exists.
\end{theorem}

\begin{proof}
Let $X = \{x_1, \dots , x_k\}$ be a $(v,k)$-MGR, where 
$x_1 < \cdots < x_k$. 
The differences $x_i - x_j$ ($i \neq j$), evaluated in $\zed_{v}$, are all distinct. 
Therefore, if we develop $X$ through the group 
$\zed_{v}$, the resulting set $\mathcal{B}$ of $v$ blocks contain every pair of points at most once.

Now, suppose we further assume that the $(v,k)$-MGR is resolvable. It is then easy to partition the 
$v$ blocks in $\mathcal{B}$ into $k$ parallel classes. 
Denote $v = kw$. Define \[P_0=\{X+kj \bmod v : j=0,1,\dots,w-1\}\] and for $1 \leq i \leq k-1$, let
\[P_i= \{ B+i \bmod v: B \in P_0 \}.\]
In this way, $\mathcal{B}$ is partitioned into $k$ parallel classes, each containing $v/k = w$ blocks, because $X$ contains one point from each residue class modulo $k$. 
\end{proof}

We give detailed proofs of the converse statements now.

\begin{theorem}\label{cyclic}
The block-set $\mathcal{B}$ of a cyclic $(v,k)$-configuration is necessarily the set of all the 
translates of a $(v,k)$-MGR.
\end{theorem}

\begin{proof}
Let $B$ be a block of $\mathcal{B}$ that is not a coset of a subgroup of $\zed_v$. Such a block obviously exists
otherwise $\mathcal{B}$ would have size at most $v/k$, which is absurd unless we are in the trivial case where $k=1$. 
Let $S$ be the stabilizer of $B$ under the action of $\zed_v$, so  we have $B=A+S$ for a suitable set $A$ 
of ${k/|S|}$ distinct representatives for the cosets of $S$ in $\zed_v$. In view of the choice of $B$, it cannot be the case that $|S|=k$.

If $1<|S|<k$, then $A$ should have at least two elements, say $a$, $a'$, and $S$ should have at least one non-zero element $s$.
Then we see that $B$ and $B+(a'-a)$ are distinct subsets of Orb$(B)$ both containing the pair $\{a',a'+s\}$, which is impossible. 
We conclude that $S$ has size 1, so Orb$(B)$ has size $v$. This means that $\mathcal{B}$ is the set of all translates of $B$.
Finally, note that the list of differences obtained from the pairs of points in $B$ cannot contain repeated elements, otherwise some pairs of points would
occur in more than one block.
\end{proof}

\begin{theorem}
\label{cyclic2}
A cyclic resolvable $(v,k)$-configuration is necessarily generated by a 
$(v,k)$-RMGR.
\end{theorem}

\begin{proof}
If $(V,\mathcal{B},\mathcal{R})$ is a cyclic resolvable $(v,k)$-configuration, then its underlying $(v,k)$-configuration $(V,\mathcal{B})$ is cyclic.
So, by Theorem \ref{cyclic}, $\mathcal{B}$ is the set of all the translates of a block $B$ that is a $(v,k)$-MGR. Let $P$ be the parallel class of $\cal R$ containing $B$ 
and let $H$ be its stabilizer under the action of $\zed_v$. Obviously, a translate $B+t$ of $B$ is contained
in $P$ if and only if $t\in H$. Thus, considering that the blocks of $P$ partition $\zed_v$, we deduce that $H$ has order ${v/k}$ and that $B$ is
a complete system of representatives for the cosets of $H$ in $\zed_v$, i.e., the elements of $B$ are pairwise distinct modulo $k$. 
\end{proof}

\begin{example}
\label{9108.ex}
It can be verified that the set \[X = \{0,1,13,32,34,39,42,56,62\}\]
is a $(108,9)$-RMGR.
That is, the elements in $X$ cover all the residue classes modulo $9$ 
and the differences of pairs of elements in $X$ are 
distinct elements in $\zed_{108}$. Thus, from Theorem \ref{cyclic2},  
$X$ gives rise to a 
cyclic resolvable $(108,9)$-configuration whose blocks are generated from $X$ by developing them modulo $108$.
\end{example}

\begin{example}
\label{11165.ex}
We have a few additional examples of RMGR,  all of which were found by computer searches.
\begin{itemize}
\item The set \[X = \{0, 1, 6, 21, 24, 52, 60, 62, 69, 136, 152\}\]
is a $(165,11)$-RMGR.
\item The set \[X = \{0, 1, 12, 15, 32, 34, 50, 55, 101, 108, 137, 174, 178\}\]
is a $(234,13)$-RMGR.
\item The set \[X = \{0, 1, 3, 10, 15, 21, 43, 87, 124, 155, 187, 206, 214\}\]
is a $(260,13)$-RMGR.
\end{itemize}
\end{example}

\begin{lemma}
\label{double.lem}
An RGR$(k,L)$ is a $(kw,k)$-RMGR if $kw \geq 2L+1$.
\end{lemma}

\begin{proof}
Let $X$ denote the RGR$(k,L)$, where $X = \{x_1, \dots , x_k\}$ and 
$x_1 < \cdots < x_k$. We can assume that $x_1 = 0$ and $x_k = L$. We consider $X$ as a subset of $\zed_{kw}$, where $kw \geq 2L+1$. The differences $x_i - x_j$ ($i \neq j$), evaluated in $\zed_{kw}$, are all distinct, because $kw \geq 2L+1$.
Also, it is clear that the resolvability property of the Golomb ruler is preserved in the modular setting, provided that $k$ divides the modulus.
\end{proof}

\begin{corollary}
\label{RGRtoPDP}
Suppose there is an RGR$(k,L)$. 
Then there exists a cyclic resolvable $(kw,k)$-configuration whenever
$kw \geq 2L+1$.
\end{corollary}

\begin{example}
\cite{stinson}
An RGR$(5,14)$, namely $\{ 0,1,8,12,14\}$, is presented in Table \ref{smallk.tab}. 
This is a $(30,5)$-RMGR since $5 \times 6 \geq 2 \times 14+1 = 29$. By developing the base block $\{ 0,1,8,12,14\}$ modulo $30$, we end up with a cyclic resolvable $(30,5)$-configuration, consisting of $30$ blocks of size $5$ that can be partitioned into five parallel classes of size $6$:

\[
\begin{array}{l|l|l}
0,1,8,12,14 & 1,2,9,13,15 & 2,3,10,14,16 \\
5,6,13,17,19 & 6,7,14,18,20 & 7,8,15,19,21 \\ 
10,11,18,22,24 & 11,12,19,23,25 & 12,13,20,24,26 \\ 
15,16,23,27,29 & 16,17,24,28,0 & 17,18,25,29,1 \\ 
20,21,28,2,4 & 21,22,29,3,5 & 22,23,0,4,6 \\
25,26,3,7,9 & 26,27,4,8,10 & 27,28,5,9,11  
\end{array}
\]

\[
\begin{array}{l|l}
3,4,11,15,17 & 4,5,12,16,18\\
8,9,16,20,22 & 9,10,17,21,23\\
13,14,21,25,27 & 14,15,22,26,28 \\
18,19,26,0,2 & 19,20,27,1,3\\
23,24,1,5,7 & 24,25,2,6,8 \\
28,29,6,10,12 & 29,0,7,11,13
\end{array}
\]
\end{example}

\bigskip

Here is a general existence result that holds for all $k$.

\begin{theorem}
Suppose $k \geq 3$ and $w\geq k^2$. Then there exists a cyclic resolvable $(wk,k)$-configuration.
\end{theorem}

\begin{proof}
From Theorem \ref{generalk.thm}, there is an RGR$(k,(k+1)(k-2)^2/2)$.
If we take $L = (k+1)(k-2)^2/2$, then it is straightforward to verify that $2L+1 \leq k^3$.
So the desired result follows immediately from Corollary \ref{RGRtoPDP}.
\end{proof}

Finally, we investigate the existence of resolvable $(v,k)$-configurations for small values of $k$.
The cases $k = 3,4$ and $5$ are handled by Theorem \ref{k=345.thm}.  We now consider the cases where $6 \leq k \leq 13$ in detail.
We have the following results.

\begin{table}[b]
\caption{Applications of Corollary \ref{RGRtoPDP} to construct cyclic resolvable $(v,k)$-configurations}
\label{PDP.tab}
\begin{center}
\begin{tabular}{r|c|l}
$k$ & RGR$(k,L)$ & \multicolumn{1}{|c}{cyclic resolvable $(v,k)$-configurations}  \\ \hline
6 & RGR$(6,20)$  & $w \geq \lceil 41/6\rceil = 7$ \\
7 & RGR$(7,31)$  & $w \geq \lceil 63/7\rceil = 9$ \\
8 & RGR$(8,45)$  & $w \geq \lceil 91/8\rceil = 12$ \\
9 & RGR$(9,58)$  & $w \geq \lceil 117/9 \rceil = 13$ \\
10 & RGR$(10,69)$  & $w \geq \lceil 139/10\rceil = 14$ \\
11 & RGR$(11,87)$  & $w \geq \lceil 175/11\rceil = 16$ \\
12 & RGR$(12,107)$  & $w \geq \lceil 215/12\rceil = 18$ \\
13 & RGR$(13,132)$  & $w \geq \lceil 265/13\rceil = 21$ \\
\end{tabular}
\end{center}
\end{table}

\begin{theorem}
\label{k=6789.thm}
For $6 \leq k \leq 13$, a resolvable $(v,k)$-configuration exists if and only if
$w \geq k$  and $(k,w) \neq (6,6)$ or $(10,10)$, 
with the following nine possible exceptions:
\begin{align*}
(k,w) \in  \{ & (9,10), (10,12), (11,12), (11,14),
(12,12), \\ &(12,14), (12,15) ,  (13,14), (13,15)  \}.
\end{align*}
\end{theorem}

\begin{proof}
We begin by listing applications of Corollary \ref{RGRtoPDP} in Table \ref{PDP.tab}. 
These applications make use of the optimal RGR$(k,L)$ presented in Table \ref{smallk.tab}.

For each $k \in \{6, \dots , 13\}$, there remain several values of $w$ to consider. We have already noted that
$w \geq k$ is a necessary condition for existence of a 
resolvable $(v,k)$-configuration. See Table \ref{PDP2.tab}
for existence and nonexistence results for the remaining ordered pairs $(k,w)$. 
The ``?'' entries in Table \ref{PDP2.tab} are the possible exceptions.
\end{proof}

\begin{table}[tb]
\caption{Existence of resolvable $(kw,k)$-configurations}
\label{PDP2.tab}
\begin{center}
\begin{tabular}{r|c|c|c}
$k$ & $w$ & existence & authority  \\ \hline
6 & 6 & no & Theorem \ref{AP.thm}\\ \hline
7 & 7,8 & yes & Corollary \ref{MOLS.cor}\\\hline
8 & 8,9,11 & yes & Corollary \ref{MOLS.cor}\\
8 & 10 & yes & Example \ref{810.ex}\\\hline
9 & 9,11 & yes & Corollary \ref{MOLS.cor}\\
9 & 12 & yes & Example \ref{9108.ex}\\
9 & 10 & ? & \\\hline
10 & 10 & no & Theorem \ref{AP.thm}\\
10 & 11,13 & yes & Corollary \ref{MOLS.cor}\\
10 & 12 & ? & \\\hline
11 & 11,13 & yes & Corollary \ref{MOLS.cor}\\
11 & 15 & yes & Example \ref{11165.ex}  \\
11 & 12,14 & ? & \\\hline
12 & 13,16,17 & yes & Corollary \ref{MOLS.cor}\\
12 & 12,14,15 & ? & \\\hline
13 & 13,16,17,19 & yes & Corollary \ref{MOLS.cor}\\
13 & 18,20 & yes & Example \ref{11165.ex} \\
13 & 14,15 & ? & \\\hline
\end{tabular}
\end{center}
\end{table}

\begin{remark}
Note that, in Theorem \ref{k=6789.thm}, we do not claim that all the constructed configurations are cyclic.
\end{remark}

\subsection{Noncyclic Groups}

Noncyclic groups can also be studied as a possible way to generate resolvable $(v,v,k,k)$-configurations.
We have shown as a result of an exhaustive search  that there does not exist a cyclic resolvable $(80,8)$-configuration. However, it turns out that we can generate a 
resolvable $(80,8)$-configuration from a suitable base block in the abelian group $\zed_8 \times \zed_{10}$.

It will be useful to define a generalization of modular Golomb rulers. Let $G$ be a finite additive group (not necessarily abelian) and let $H$ be a subgroup of $G$. Suppose $|H| = w$ and $|G| = kw$.
A \emph{$(G,H)$-group Golomb ruler} (or \emph{$(G,H)$-GGR}, for short) is a subset $X$ of $G$ of size $k$ that satisfies the following properties:
 \begin{enumerate}
 \item the differences obtained from pairs of elements in $X$ are all distinct, and
 \item $X$ is a complete set of representatives of the left cosets of $H$ in $G$.
 \end{enumerate}

The following lemma is a simple consequence of the definitions.
\begin{lemma}
 Suppose $G = \zed_{kw}$ and  let $H$ be the unique subgroup of $G$ of order $w$.
 Then $X$ is a $(G,H)$-GGR if and only if $X$ is a $(kw, k)$-RMGR.
\end{lemma}
 
Group Golomb rulers can also be used to construct resolvable symmetric configurations.
 
 \begin{theorem}
 \label{G-ruler.thm} Suppose $G$ is a finite group of order $kw$ and $H$ is a subgroup of order $w$, and 
 suppose $X$ is a $(G,H)$-GGR. Then there exists a resolvable $(kw,k)$-configuration.
 \end{theorem}
 
 \begin{proof} We construct $kw$ blocks from $X$. For any $g \in G$, define
 the block \[X+g = \{ x+g : x \in X\}.\] The set of $kw$ blocks $\{X+g : g \in G\}$ clearly does not contain any pair of points more than once, so it is a symmetric configuration.
 
We describe a set of parallel classes that form a resolution.
First, define 
\[ P_0 = \{ X + h : h \in H\}.\]
It is easy to see that $P_0$ is a parallel class. 
We obtain the other parallel classes by letting $G$ act on $P_0$.
The orbit of $P_0$ under this action is a resolution of the blocks.
 \end{proof}

 \begin{example}
 \label{810.ex}
 Let $k = 8$ and $w = 10$.
 Let $G = \zed_8 \times \zed_{10}$ and let $H = \{0,4\} \times \{0,2,4,6,8\}$. $G$ is an abelian group of order $80$ and $H$ is a subgroup of order $10$.
Let \[X=\{(0,0),(0,1),(1,0),(5,1),(2,4),(2,7),(3,2),(7,9)\}.\]
It is straightforward to verify that the conditions of Theorem \ref{G-ruler.thm} are satisfied. Therefore there exists a resolvable $(80,8)$-configuration.
 \end{example}
 
 \begin{example}
 We give an example of a GGR in a non-abelian group.
Let $G=A_4$ be the alternating group of degree $4$, i.e., the
group of even permutations defined on the set $\{1,2,3,4\}$. 
Let 
\[H=\{\mathrm{id}, (12)(34), (13)(24), (14)(23)\}.\]
Then $H$ is subgroup of $G$ of order $4$.
The subset \[X = \{\mathrm{id}, (123), (124)\}\] satisfies condition 1.
Indeed, the set of ``differences'' (actually, the permutations $\pi \rho^{-1}$, for $\pi, \rho \in X$, since the group operation is written multiplicatively) are the permutations in the set 
\[ \{(123), (132), (124), (142), (234), (243)\}.\]
Also, $X$ is a
complete set of representatives for the left cosets
of $H$ in $G$.
Thus $X$ generates a ``non-abelian'' resolvable $(12,3)$-configuration whose
resolution is the $A_4$-orbit of the parallel class
$P = \{Xh : h \in H\}$.
\end{example}
 
\section{Progressive Dinner Parties}
\label{PDP.sec}

Resolvable Golomb rulers can be used to construct a certain 
kind of progressive dinner party (or PDP) that is defined in \cite{stinson}.
The objective is to design a dinner party for $v$ couples that satisfies the following conditions:
\begin{enumerate}
\item each course of a $k$-course dinner is attended by $k$ couples, at $v/k$ different houses,
\item no two couples dine together at more than one course of the meal, and
\item each couple hosts exactly one course of the meal.
\end{enumerate}
Therefore, we define  
a PDP$(k,v)$ to be a set of blocks of size $k$, defined on a set of $v$ points, which satisfies the following properties:
\begin{enumerate}
\item The blocks can be partitioned into $k$ parallel classes, each consisting of $v/k$ disjoint blocks. (Hence, there are a total of $v$ blocks and we require $v \equiv 0 \bmod k$.)
\item No pair of points occurs in more than one block.
\item There is a bijection $h : \mathcal{B} \rightarrow X$ such that  
$h(B) \in B$ for all $B \in \mathcal{B}$. 
\end{enumerate}
It is shown in \cite{stinson} that the third condition of the above definition always holds when the first two conditions hold. So we have the following  characterization of PDP$(k,v)$ in terms of resolvable symmetric configurations.
\begin{theorem}
\label{PDP.thm}
A PDP$(k,v)$ is equivalent to a resolvable $(v,k)$-configuration. 
\end{theorem}
In view of Theorem \ref{PDP.thm}, all existence results for resolvable symmetric configurations from Sections
\ref{RSC.sec} and \ref{CC.sec} automatically carry over to progressive dinner parties.

\section{Discussion}
\label{summary.sec}

We have introduced a new type of Golomb ruler, a resolvable Golomb ruler, in this paper. We were originally motivated by an application to the construction of progressive dinner parties, which were defined in \cite{stinson}. However, resolvable Golomb rulers seem to be interesting combinatorial structures in their own right.

We have observed that progressive dinner parties are equivalent to resolvable symmetric configurations. If we drop the ``symmetric'' condition, we might instead consider the problem of constructing resolvable configurations in which the number blocks is as large as possible. Suppose the block size is $k$ and the number of points is $kw$. It is clear that the maximum number of parallel classes is  $\lfloor (kw-1)/(k-1) \rfloor$  and hence the total number of blocks is at most $w \lfloor (kw-1)/(k-1) \rfloor$.

For example, suppose $k = 5$ and $w  = 11$. Then the maximum number of parallel classes is 
 $\lfloor 54/4 \rfloor = 13$. It is possible to construct a resolvable configuration
 with $10$ parallel classes by developing  the following two base blocks through $\zed_{55}$:  
 \begin{center}
 $A = \{0,1,17,53,24\}$ and
$B=\{0,6,27,18,14\}$.\end{center} 
Each of $A$ and $B$ contain one element from each residue class modulo $5$. Therefore, we can partition the set of blocks into ten parallel classes, consisting of the five distinct translates of
$\{A+5i \bmod 55: i=0,1,\dots , 11\}$ and the five distinct translates of
$\{B+5i \bmod 55: i=0,1,\dots , 11\}$.

\section*{Acknowledgements} We would like to thank Shannon Veitch for assistance with programming.

\end{document}